\documentclass[a4paper,american,11pt]{amsart}

\usepackage{amsmath}
\usepackage{amssymb}
\usepackage{caption}
\usepackage{listings}
\usepackage{subcaption}
\usepackage{amsfonts}
\usepackage{amscd}
\usepackage{wrapfig}
\input xy
\xyoption{all}
\usepackage{microtype} 
\usepackage{parskip}   
\usepackage{booktabs}  
\usepackage{array}     
\usepackage{amsmath}   
\usepackage{float}     

\usepackage[usenames]{color} 
\usepackage{colortbl}
\definecolor{cobalt}{rgb}{0.0, 0.28, 0.67}

\usepackage{hyperref}

\newtheorem{definition}{Definition}
\newtheorem{lemma}{Lemma}
\newtheorem{theorem}{Theorem}
\newtheorem{conjecture}{Conjecture}

\newcommand\ZZ{\mathbb Z}
\newcommand\FF{\mathbb F}

\newcommand\QQ{\mathbb Q}

\begin{document}



\author{Nikita Kalinin}
\title{Wolstenholme's theorem over Gaussian integers}
\address{Guangdong Technion Israel Institute of Technology (GTIIT),
241 Daxue Road, Shantou, Guangdong Province 515603, P.R. China,  Technion-Israel Institute of Technology, Haifa, 32000, Haifa District, Israel, nikaanspb@gmail.com}
\maketitle

\begin{abstract}
This paper establishes an extension of Wolstenholme's theorem to the ring of Gaussian integers $\mathbb{Z}[i]$. For a prime $p > 7$, we prove that the sum $S_p$ of inverses of Gaussian integers in the set $\{n+mi \mid 1 \leq n, m \leq p-1, \gcd(p, mi+n)=1\}$ satisfies the congruence $S_p \equiv 0 \pmod{p^4}$. We further generalize this result to higher-power sums $S_p^{(k)}$, demonstrating structured divisibility patterns modulo powers of $p$. We propose some conjectures generalising the connections between classical Wolstenholme's theorem and binomial coefficients. Special cases and irregularities for small primes ($p \leq 1000$) are explicitly computed and tabulated.

\end{abstract}

\section{Introduction}
The classical Wolstenholme's theorem \cite{wolstenholme1862certain} states that for any prime $p \geq 5$,

$$1+\frac{1}{2}+\frac{1}{3}+\dots + \frac{1}{p-1}\equiv 0 \pmod {p^2}.$$

This result has been generalized to various settings \cite{zhao2007bernoulli,mevstrovic2013some,lucas1878congruences,glaisher1900congruences,alkan1994variations,carlitz1954note,mcintosh1995converse, andrews1999q, mestrovic2011wolstenholme, sun2022series, sun2000congruences}. 

Let $i^2 = -1$ and $p$ be an integer prime number. Our main result generalizes Wolstenholme's theorem to Gaussian integers:

\begin{theorem}[Wolstenholme's Congruence for Gaussian Integers]
\label{th1}
For a prime $p > 5$, the sum $S_p$ of reciprocals of Gaussian integers modulo $p$ satisfies:
\begin{equation}
\label{eq1}
S_p = \sum_{\substack{1 \leq n, m \leq p-1 \\ (p, m^2+n^2)=1}} \frac{1}{n+mi} \equiv 0 \pmod{p^4}.
\end{equation}
\end{theorem}

\subsection{Plan of the paper} In Section~\ref{sec2} we prove this theorem. Section~\ref{sec3} contains a generalisation of the above congruence for sums $S_p^{{k}}$ of the $k$-th powers $\frac{1}{(mi+n)^k}$ of fractions for $k=2,3,4$ and a conjecture about the general case is stated. We also list cases when we have a congruence modulo higher power than expected for $p<1000$ and $k\leq 12$. Section~\ref{sec4} contains discussion, speculations, and more conjectures.

\subsection{Motivation} Our initial aim was to establish an analogue of Wolstenholme's theorem for finite fields with $p^k$ elements. The classical Wolstenholme's theorem can be formulated as follows. Let $p$ be a prime integer number and $\ZZ\to\ZZ_p= \mathbb F_p$ be the map of taking residues mod $p$. Consider the invertible elements  of $\mathbb F_p$ and lift them in $\ZZ$ as numbers $1,2,\dots, p-1$. Take their inverses in $\QQ$, compute the sum $$\sum_{k=1}^{p-1}\frac{1}{k}=\frac{A}{B},$$ then the denominator $B$ is coprime with $p$ and the numerator $A$ is divisible by $p$ (since $\sum_{k=1}^{p-1} \frac{1}{k}=0$ if we consider the elements $k\in \FF^{*}_p$). Wolstenholme's theorem states that the numerator $A$ is divisible by a higher than expected power of $p$, namely $p^2$.

For a finite field $\mathbb F^*_{p^k}$ of the characteristic $p$ with $p^k$ elements, a natural approach is to lift the sum $$\sum\limits_{a\in \mathbb F^*_{p^k}} \frac{1}{a}=0\in \mathbb F^*_{p^k}$$ to a ring of characteristic $0$, compute the sum of fractions in this ring  and then establish divisibility of the numerator by a power of $p$ (again, divisibility by $p^1$ is automatic). 

After some unsuccessful experiments, we realized $\mathbb F_{p^2}$ as $\mathbb F_p[i]$ for the case $p\equiv 3 \pmod 4$, then found and proved the above congruence. Note that $\mathbb F^*_{p^2}$ contains $p^2-1$ elements but in Theorem~\ref{th1} we compute the sum over $(p-1)^2$ elements (for $p=4k+3$) or $(p-1)^2-(p-1)$ elements (for $p=4k+1$).  It would be interesting to establish similar congruences for extensions of $\mathbb F_p$ of higher degree.

\section{Proof of Theorem~\ref{th1}} 
\label{sec2}

%
%
%
%
%
%

Before proceeding with the proof, we require a lemma. 

\begin{lemma}\label{lem} For each prime number $p>5$ the following congruence holds:

$$\sum_{\substack{1\leq n,m \leq {p-1}\\(p,m^2+n^2)=1}}\frac{(m^4-6m^2n^2+n^4)}{(m^2+n^2)^4}\equiv 0 \pmod p.$$

\end{lemma}

\begin{proof}
We need to show that both sums below vanish modulo $p$:

$$\sum_{\substack{1\leq n,m \leq {p-1}\\(p,m^2+n^2)=1}}\frac{m^2n^2}{(m^2+n^2)^4}, \text{ and} $$ $$ \sum_{\substack{1\leq n,m \leq {p-1}\\(p,m^2+n^2)=1}}\frac{(m^4+2m^2n^2+n^4)}{(m^2+n^2)^4} =  \sum_{\substack{1\leq n,m \leq {p-1}\\(p,m^2+n^2)=1}}\frac{1}{(m^2+n^2)^3}.$$

The key observation is that modulo $p$ for $m^2+n^2\ne 0 \pmod p$, we have $$\frac{1}{m^2+n^2} \equiv (m^2+n^2)^{p-2} \pmod p.$$ Moreover, the power sums $\sum_{n=1}^{p-1} n^q \equiv 0 \pmod{p}$ vanish unless $q$ is divisible by $p-1$. 

Thus, if $p>5$ 

$$\sum_{\substack{1\leq n,m \leq p-1\\(p,m^2+n^2)=1}}\frac{m^2n^2}{(m^2+n^2)^4} \equiv \sum_{\substack{1\leq n,m \leq {p-1}\\(p,m^2+n^2)=1}}(m^2n^2)(m^2+n^2)^{p-5}\equiv $$

$$\equiv \sum_{1\leq n,m \leq {p-1}}(m^2n^2)(m^2+n^2)^{p-5} \equiv \sum_{0\leq n,m \leq {p-1}}(m^2n^2)(m^2+n^2)^{p-5}  \pmod p.$$

When we expand the parenthesis, we have terms of the form $$n^tm^s, y+s=2p-6$$ so it is impossible that $p-1$ divides both $s$ and $t$ unless $2p-6=0$, i.e. $p=3$. If, for example, $p-1$ does not divide $s$ then for each fixed  $m$ we have $$\sum_{1\leq n\leq p-1} n^sm^t\equiv  0\pmod p.$$ 

Thus, for each term $n^sm^t, s+t = 2p-6$ we have 
$$\sum\limits_{1\leq n,m\leq p-1} n^sm^t\equiv 0 \pmod p.$$

Similarly, 

$$\sum_{\substack{1\leq n,m \leq {p-1}\\(p,m^2+n^2)=1}}\frac{1}{(m^2+n^2)^3} \equiv \sum_{1\leq n,m \leq p-1}(m^2+n^2)^{p-4} \pmod p.$$

The terms in this sum are $n^tm^s, y+s=2p-8$ and we proceed as above.

\end{proof}

\begin{definition}[Conjugate Tuples]
For distinct indices $n\ne m$, define the following $8$-tuple $T(m,n,k)$ as  
\begin{multline}\label{eq2}T(m,n,k)=\frac{1}{(n+im)^k}+\frac{1}{(p-n+im)^k}+\frac{1}{(n+i(p-m))^k}+\\
+\frac{1}{(p-n+i(p-m))^k}+\frac{1}{(m+in)^k}+\frac{1}{(p-m+in)^k}+\\
+\frac{1}{(m+i(p-n))^k}+\frac{1}{(p-m+i(p-n))^k}.\end{multline}
\end{definition}

\begin{proof}[Proof of Theorem~\ref{th1}]
Consider a particular element  $\frac{1}{n+mi}$ in \eqref{eq1}. Consider the corresponding $8$-tuple $T(m,n,1)$. All such tuples are parametrised by pairs $(n,m)$ with $1\leq n<m\leq \frac{p-1}{2}$. Direct computation shows that the product of the denominators of fractions in $T(m,n,1)$ is 
 $$
 m^{8}+4m^{6}n^{2}+6m^{4}n^{4}+4m^{2}n^{6}+n^{8}+ p(\dots)$$
and the numerator of $T(n,m,1)$ is
      
$$\left(-2i+2\right)p^3 (m^{4}-6m^{2}n^{2}+n^{4})+ p^4(\dots).$$

Note that the numerator is already divisible by $p^3$  (see Section~\ref{sec_tnm} for explicit expansions).

Now we can rewrite 

$$S_p = \sum_{\substack{1 \leq n, m \leq p-1 \\ (p, m^2+n^2)=1}} \frac{1}{n+mi} = \frac{1}{8}\sum_{\substack{1 \leq n, m \leq p-1 \\ (p, m^2+n^2)=1}} T(n,m,1) \equiv  $$
$$\equiv p^3\sum_{\substack{1\leq n,m \leq {p-1}\\(p,m^2+n^2)=1}}\frac{(m^4-6m^2n^2+n^4)}{(m^2+n^2)^4} \pmod{p^4}.$$

Theorem~\ref{th1} follows now from Lemma~\ref{lem}.

%
%
%
%
%
%

\end{proof}

\section{Higher powers}
\label{sec3}
 Having established the base case, we now consider higher powers.

 \begin{theorem}
 \label{th2}
For a prime $p>5$ and $k\geq 1$, define:
\begin{equation}
S_p^{(k)} = \sum_{\substack{1\leq n,m\leq p-1\\ (p,m^2+n^2)=1}} \frac{1}{(n+mi)^k}
\end{equation}
Then the following congruences hold:
\begin{enumerate}
\item $S_p^{(1)} \equiv 0 \pmod{p^4}$ 
\item $S_p^{(2)} \equiv 0 \pmod{p^3}$ 
\item $S_p^{(3)} \equiv 0 \pmod{p^2}$ 
\item $S_p^{(4)} \equiv 0 \pmod{p}$
\end{enumerate}
\end{theorem}

%

\begin{proof}
(1) is Theorem~\ref{th1}, it is already proven. For (2),(3), and (4), the proofs follow the similar lines. Consider $8$-tuples $T(m,n,k), k=2,3,4$. We perform the computations in Macaulay2; the code is in Section~\ref{sec_macaulay}. 

In (2)  the denominator of $T(m,n,2)$ modulo $p$ is $(m^2+n^2)^8$ and the numerator is $$12ip^{2}\left(m^{2}+n^{2}\right)^{4}\left(m^{2}-2mn-n^{2}\right)\left(m^{2}+2mn-n^{2}\right).$$

Thus it is enough to show that $$\sum \frac{\left(m^{2}-2mn-n^{2}\right)\left(m^{2}+2mn-n^{2}\right)}{(m^2+n^2)^4}\equiv 0 \pmod p,$$

and the numerator of this expression is $m^4-6m^2n^2+n^4$, hence we may use Lemma~\ref{lem}. 

In (3) the denominator of $T(m,n,3)$ modulo $p$ is $(m^2+n^2)^{12}+p...$ and the numerator is $$\left(-12i-12\right)p\cdot(m^2+n^2)^8\left(m^{2}-2mn-n^{2}\right)\left(m^{2}+2mn-n^{2}\right)$$ and the proof is the same as above.

In (4) the denominator  of $T(m,n,4)$ modulo $p$ is  $(m^2+n^2)^{16}+p..$ and the numerator is 

$$8(m^2+n^2)^{12}\left(m^{2}-2mn-n^{2}\right)\left(m^{2}+2mn-n^{2}\right)+p...$$ and the rest of the proof is the same.

\end{proof}

For $S_p^{(5)}$ the exponent of $p$ in
\begin{multline}\frac{1}{(n+im)^5}+\frac{1}{(p-n+im)^5}+\frac{1}{(n+i(p-m))^5}+\frac{1}{(p-n+i(p-m))^5}+\\
+\frac{1}{(m+in)^5}+\frac{1}{(p-m+in)^5}+\frac{1}{(m+i(p-n))^5}+\frac{1}{(p-m+i(p-n))^5}\end{multline}

is controlled by

 $$\frac{p^{3}\left(m^{4}-4m^{3}n-6m^{2}n^{2}+4mn^{3}+n^{4}\right)\left(m^{4}+4m^{3}n-6m^{2}n^{2}-4mn^{3}+n^{4}\right)}{\left(m^{2}+n^{2}\right)^{8}}$$

In the above formula, we throw away all the terms containing $p$ in the denominator and those containing $p^4$ in the numerator and then divide the numerator by $-70(i-1)$. Again, it is not difficult to prove the congruence $S^{(5)}_p\equiv 0 \pmod {p^4}$, but the proof is not identical to the proof of Theorem~\ref{th1}, and we do not see how to proceed in the general case, i.e. how to write a general formula for the coefficient behind the lowest power of $p$ in the numerator of $T(m,n,k)$. 

A straightforward calculation yields divisibility of $$S^{(5)}_{p},S^{(6)}_{p},S^{(7)}_{p},S^{(8)}_{p},S^{(9)}_{p},S^{(10)}_{p},S^{(11)}_{p},S^{(12)}_{p} $$ by $p^4, p^3,p^2,p,p^4, p^3,p^2,p$ correspondingly. Motivated by these computations, we propose the following conjecture.

\begin{conjecture}[Wolstenholme's Higher-Power Congruences for Gaussian Integers]
\label{conj:main}
For any integer $k \geq 1$ and prime $p > 17$,
the sum $S_p^{(k)}$ satisfies:
\begin{equation}
S_p^{(k)}=\sum_{\substack{1\leq n,m\leq p-1\\(p,n^2+m^2)=1}}\frac{1}{(mi+n)^{k}} \equiv 0 \pmod{p^{m(k)}}, \quad 
m(k) = \begin{cases}
4 & \text{if } k \equiv 1 \pmod{4} \\
3 & \text{if } k \equiv 2 \pmod{4} \\
2 & \text{if } k \equiv 3 \pmod{4} \\
1 & \text{if } k \equiv 0 \pmod{4}
\end{cases}
\end{equation}
One can write $m(k)$ as $m(k)=4-(k-1)\pmod 4.$ 
\end{conjecture}
A lower bound of the form $p > ck$ for some absolute constant $c$ may be required, but our experiments do not suggest this.

\begin{table}[H]
\centering
\caption{Irregular cases for $S_p^{(k)} \pmod{p^m}$ with $p < 1000$ and $1 \leq k \leq 12$.}
\label{irregular}
\begin{tabular}{lllll}
\toprule
Power $k$ & Primes $p$ & Expected & Observed & Type \\
\midrule

1 & 3,5 & $p^4$ & $p^3$ & Weaker \\
 & 31,37 & $p^4$ & $p^5$ & Stronger \\

\midrule
2 & 3 & $p^3$ & $p^3$ & Expected \\
 & 5 & $p^3$ & $p^2$ & Weaker \\
 & 31,37 & $p^3$ & $p^4$ & Stronger \\
\midrule
3 & 3 & $p^2$ & $p^2$ & Expected \\
 & 5 & $p^2$ & $p^1$ & Weaker \\
 & 31,37 & $p^2$ & $p^3$ & Stronger \\
\midrule
4 & 3,5 & $p^1$ & None & None \\
 & 31,37 & $p^1$ & $p^2$ & Stronger \\
\midrule
5 & 3 & $p^4$ & $p^3$ & Weaker \\
 & 7,67,877 & $p^4$ & $p^5$ & Stronger \\
\midrule
6 & 7,67,877 & $p^3$ & $p^4$ & Stronger \\
\midrule
7 & 3,5 & $p^2$ & $p^1$ & Weaker \\
  & 7,67,877 & $p^2$ & $p^3$ & Stronger \\
\midrule
8 & 3,5 & $p^1$ & None & None \\
  & 67,877 & $p^1$ & $p^2$ & Stronger \\
\midrule
9 & 7,13 & $p^4$ & $p^3$ & Weaker \\
 & 11 & $p^4$ & $p^5$ & Stronger \\
\midrule
10 & 3 & $p^3$ & $p^2$ & Weaker \\
 & 11 & $p^3$ & $p^4$ & Stronger \\
\midrule
11 & 3,5,7,13 & $p^2$ & $p^1$ & Weaker \\
 & 11 & $p^2$ & $p^3$ & Stronger \\
\midrule
12 & 3,5,7,13 & $p^1$ & None & None \\
\bottomrule
\end{tabular}
\end{table}

Table~\ref{irregular} enumerates all observed irregularities for $p<1000, 1\leq k\leq 12$.
\begin{itemize}
    \item \(\text{Expected}\): The conjectured congruence \(S_p^{(k)} \equiv 0 \pmod{p^{4-r}}\) where \(r = k-1 \pmod{4},  r\in\{0,1,2,3\}\).
    \item \(\text{Observed}\): The actual congruence found in computations.
    \item \(\text{Type}\): 
        \begin{itemize}
            \item \textbf{Weaker}: Observed modulus is less than expected.
            \item \textbf{Stronger}: Observed modulus is greater than expected.
            \item \textbf{None}: No divisibility.
        \end{itemize}
\end{itemize}

\section{Discussion and possible directions}
\label{sec4}

\subsection{Polynomial Analogues} Wolstenholme's theorem admits the following polynomial interpretation. Let $$f(x)=\prod_{k=1}^{p-1}(x-k)=x^{p-1}+a_1x^{p-2}+\dots+a_{p-2}x+a_{p-1}.$$

Then $a_{p-2}\equiv 0 \pmod {p^2}.$ Indeed, $-a_{p-2}/a_{p-1}$ is the sum of the inverses of the roots of $f(x)$. 

For Gaussian integers define $r=(p-1)^2$ and  $$g_p(x) = \prod_{\substack{1\leq n,m\leq p-1\\(p,n^2+m^2)=1}} (x-(n+mi))=x^{(p-1)^2}+a_1x^{(p-1)^2-1}+\dots + a_{r-1}x+a_{r}.$$

Then $$\frac{-a_{r-1}}{a_{r}}=\sum_{\substack{1\leq n,m\leq p-1\\(p,n^2+m^2)=1}} \frac{1}{n+mi}$$

The polynomial $g_p$ has nontrivial coefficients. For example, consider the following polynomials $g_p(x) \pmod p$

$$g_{11}(x)\equiv x^{100}+x^{80}+x^{60}+x^{40}+x^{20}+1 \pmod {11}$$

$$g_{19}(x)\equiv x^{324}+x^{288}+x^{252}+x^{216}+x^{180}+x^{144}+x^{108}+x^{72}+x^{36}+1 \pmod {19}$$

$$g_{13}(x)\equiv x^{120}+3x^{108}+6x^{96}-3x^{84}+2x^{72}-5x^{60}+2x^{48}-3
      x^{36}+6x^{24}+3x^{12}+1 \pmod {13}$$

\begin{align*}
g_{17}(x) &\equiv  x^{224} + 3x^{208} + 6x^{192} - 7x^{176} - 2x^{160} \\
&\quad + 4x^{144} - 6x^{128} + 2x^{112} - 6x^{96} + 4x^{80} \\
&\quad - 2x^{64} - 7x^{48} + 6x^{32} + 3x^{16} + 1 \pmod {17}
\end{align*}

Our findings lead to the following conjectural pattern.
\begin{conjecture}
If $p=4k+3>3$ then $$g_p(x)  \equiv 1+x^{2(p-1)}+x^{4(p-1)}+\dots+ x^{(p-1)^2}\equiv \frac{1-x^{p^2-1}}{1-x^{2(p-1)}} \pmod p.$$ 

If $p=4k+1>5$ then $$g_p(x) \equiv 1 + b_1x^{p-1}+b_2x^{2(p-1)}+\dots+x^{(p-1)(p-3)} \pmod p.$$ Is there any closed formula for the coefficients $b_k\equiv a_{k(p-1)}\pmod p$ ?
\end{conjecture}

\subsection{Binomial coefficients for Gaussian integers} By Vieta's theorem, Theorem~\ref{th2} implies  the divisibility of the coefficients of $g_p(x)$ behind $x^1,x^2,x^3,x^4$ by the powers of $p$, i.e. the congruences $a_{r-k}\equiv 0 \pmod {p^{5-k}}$ for $k=1,2,3,4$. 

Substituting $x=pA+ipB, A,B\in\mathbb Z$ into $g_p(x)$ we get
\begin{theorem}
\label{th5}
For a prime $p>5$ and  $A,B\in\ZZ_{>1}$ we have

$$\prod_{\substack{1\leq n,m\leq p-1\\(p,n^2+m^2)=1}} (pA+ipB-(n+mi))\equiv \prod_{\substack{1\leq n,m\leq p-1\\(p,n^2+m^2)=1}}(n+mi) \pmod {p^5}.$$
\end{theorem}

This may be seen as a $\mathbb{Z}[i]$-analogue of  
\begin{equation*}  
(np-1)\cdot(np-2)\cdot \ldots \cdot (np-p+1) \equiv (p-1)! \pmod{p^3},  
\end{equation*}  
which is equivalent to $\binom{np-1}{p-1} \equiv 1 \pmod{p^3}$. 

Thus, one could define a Gaussian integer analog of the binomial coefficients.

\begin{definition}
For integer numbers $A\geq C\geq 1,B\geq D\geq 1$ define the \emph{Gaussian binomial coefficient}:
\begin{equation}
\left[\substack{A+Bi \\ C+Di}\right] := \frac{\prod\limits_{\substack{0\leq n\leq  C-1,\\ 0\leq m\leq  D-1}} (A+Bi-(n+mi))}{\prod\limits_{\substack{1\leq n\leq C, \\ 1\leq m\leq D}} (n+mi)}
\end{equation}
\end{definition}


Unfortunately, such binomial coefficients are not Gaussian integers in general and we do not know much about their properties. However we can reformulate Theorem~\ref{th5} as

\begin{theorem} Let $A,B \in \mathbb{Z}_{\geq 1}$ and prime $p\equiv 3 \pmod 4, p>3$. Then
\begin{equation}
\left[\substack{pA-1+(pB-1)i \\ (p-1)+(p-1)i}\right] \equiv 1 \pmod{p^5}.
\end{equation}

\end{theorem}
 
It is interesting whether a version of Lucas' theorem holds for such coefficients. 

\begin{conjecture}[Lucas congruence for Gaussian binomial Coefficients]
$$\left[\!\substack{pA+pBi \\ pC+pDi}\!\right] \equiv \left[\!\substack{A+Bi \\ C+Di}\!\right]  \pmod{p^3}.$$
\end{conjecture}

\subsection{Connections to Bernoulli Numbers}  It would be interesting to find a generalization of Glaisher's congruence \cite{glaisher1900congruences}:

$$1+\frac{1}{2}+\frac{1}{3}+\dots + +\frac{1}{p-1}\equiv -\frac{1}{3}p^2B_{p-3}\pmod {p^3},$$
where $B_{p-3}$ is a Bernoulli number; see also \cite{sun2000congruences} for generalizations. Maybe it is possible to find an explicit $p$-adic expansions by $L$-functions, such as in \cite{washington1998p}. It is interesting to investigate the Bernoulli numbers for the sums of powers of  Gaussian integers $$W^{(k)}(A+Bi)=\sum_{1\leq a\leq A, 1\leq b\leq B} (a+bi)^k.$$

\subsection{Sums for general $n+ni$}
For a (not necessarily prime) natural number $n\equiv 1,5\pmod 6$ Ludesdorf \cite{leudesdorf1888some} proved that

$$\sum_{\substack{1\leq k< n\\ (k,n)=1}}\frac{1}{k}\equiv 0 \pmod {n^2}$$ 

One can consider sums of such fractions' powers; they are curiously related to the Bernoulli numbers \cite{slavutskii1999leudesdorf,sun2000congruences, washington1998p}.

So it is natural to ask when $$S_n^{(k)}\equiv 0 \pmod {n^{4-r}}, r= (k-1)\pmod 4$$ holds for an arbitrary number $n$. No definitive pattern is apparent here; let us list some observations. Consider $S_n^{(k)}$ for $1\leq k\leq 8$.
Besides primes, for certain numbers $n=21, 26, 34, 35, 39,40,4249,52, 55,57,58, 63, 68, 74,77,78,82,\\ 84, 91,93, 104,106, 110,111, 114,116,117, 119, 121, 122, 126, 129, 133, 136, 143, \\144,145,146,147, 148,154,155,156,161, 164, ... $ all the congruences hold, the same as for the prime numbers.  If $5|n$ then it is frequent that only the divisibility for $S_n^{(5)}$ holds (among those for $1\leq k\leq 8$), and others do not.

\section{$T(n,m,1)$}
\label{sec_tnm}

The product of the denominators of fractions in $T(m,n,1)$ is 
 \begin{align*}
 m^{8}+4m^{6}n^{2}+6m^{4}n^{4}+4m^{2}n^{6}+n^{8}-4m^{7}p-4m^{6}np- \\
 -12m^{5}n^{2}p-12m^{4}n^{3}p-12m^{3}n^{4}p-12m^{2}n^{5}p-4mn^{6}p-\\
 -4n^{7}p+8m^{6}p^{2}+12m^{5}np^{2}+24m^{4}n^{2}p^{2}+24 m^{3}n^{3}p^{2}+\\
 +24m^{2}n^{4}p^{2}+12mn^{5}p^{2}+8n^{6}p^{2}-10m^{5}p^{3}-18m^{4}np^{3}-\\
 -28m^{3}n^{2}p^{3}-28m^{2}n^{3}p^{3}-18mn^{4}p^{3}-10n^{5}p^{3}+9m^{4}p^{4}+\\
 +16m^{3}np^{4} +18m^{2}n^{2}p^{4}+16mn^{3}p^{4}+9n^{4}p^{4}-6m^{3}p^{5}-\\
 -6m^{2}np^{5}-6mn^{2}p^{5}-6n^{3}p^{5}+2m^{2}p^{6}+2n^{2}p^{6}.\end{align*}

The numerator of the sum  $T(n,m,1)$ is
      
\begin{align*}\left(-2i+2\right)m^{4}p^{3}+\left(12i-12\right)m^{2}n^{2}p^{3}+\left(-2i+2\right)n^{4}p^{3}+\\
+\left(4i-4\right)m^{3}p^{4}+\left(-12i+
      12\right)m^{2}np^{4}+\left(-12i+12\right)mn^{2}p^{4}+\\
      +\left(4i-4\right)n^{3}p^{4}+\left(12i-12\right)mnp^{5}+\left(-2i+2\right)m
      p^{6}+\left(-2i+2\right)np^{6}\end{align*}

\section{Macaulay2 code}
\label{sec_macaulay}

The following \textsc{Macaulay2} code was used: 

\begin{verbatim}

R = ZZ[i]/(i^2 + 1)
R2= R[m, n, p]

--choose the power k in S_p^(k)
--k=1
k=2
--k=3
--k=4

i1=(m+i*n)^k
i2=((p-m)+i*n)^k
i3=(m+i*(p-n))^k
i4=((p-m)+i*(p-n))^k
i5=(n+i*m)^k
i6=((p-n)+i*m)^k
i7=(n+i*(p-m))^k
i8=((p-n)+i*(p-m))^k

--denominator
f= i1*i2*i3*i4*i5*i6*i7*i8
--denominator mod p
f1= sum select(terms(f), t -> (first exponents t)_2 <= 0)
factor(f1)

--numerator
g=i1*i2*i3*i4*i5*i6*i7+ i1*i2*i3*i4*i5*i6*i8+ i1*i2*i3*i4*i5*i7*i8+ 
i1*i2*i3*i4*i6*i7*i8+ i1*i2*i3*i5*i6*i7*i8+
i1*i2*i4*i5*i6*i7*i8+ i1*i3*i4*i5*i6*i7*i8+ i2*i3*i4*i5*i6*i7*i8
--the lowest power of p in the numerator is 4-k
g1=sum select(terms(g), t -> (first exponents t)_2 <= 4-k)
factor(g1)
\end{verbatim}

\section{Acknowledgements}
The author thanks Ivan Vasilyev, Fedor Petrov, and Evgeny Smirnov for inspiring discussions.

\bibliography{../bibliography.bib}

\begin{thebibliography}{10}

\bibitem{alkan1994variations}
E.~Alkan.
\newblock Variations on {W}olstenholme's theorem.
\newblock {\em The American mathematical monthly}, 101(10):1001--1004, 1994.

\bibitem{andrews1999q}
G.~E. Andrews.
\newblock $q$-analogs of the binomial coefficient congruences of {B}abbage,
  {W}olstenholme and {G}laisher.
\newblock {\em Discrete Mathematics}, 204(1-3):15--25, 1999.

\bibitem{carlitz1954note}
L.~Carlitz.
\newblock A note on {W}olstenholme's theorem.
\newblock {\em The American Mathematical Monthly}, 61(3):174--176, 1954.

\bibitem{glaisher1900congruences}
J.~W. Glaisher.
\newblock Congruences relating to the sums of products of the first $n$ numbers
  and to other sums of products.
\newblock {\em Quart. J. Math}, 31:1--35, 1900.

\bibitem{leudesdorf1888some}
C.~Leudesdorf.
\newblock Some results in the elementary theory of numbers.
\newblock {\em Proceedings of the London Mathematical Society}, 1(1):199--212,
  1888.

\bibitem{lucas1878congruences}
{\'E}.~Lucas.
\newblock Sur les congruences des nombres eul{\'e}riens et des coefficients
  diff{\'e}rentiels des fonctions trigonom{\'e}triques suivant un module
  premier.
\newblock {\em Bulletin de la Soci{\'e}t{\'e} math{\'e}matique de France},
  6:49--54, 1878.

\bibitem{mcintosh1995converse}
R.~J. McIntosh.
\newblock On the converse of {W}olstenholme's theorem.
\newblock {\em Acta Arithmetica}, 71(4):381--389, 1995.

\bibitem{mestrovic2011wolstenholme}
R.~Mestrovic.
\newblock Wolstenholme's theorem: Its {G}eneralizations and {E}xtensions in the
  last hundred and fifty years (1862--2012).
\newblock {\em arXiv preprint arXiv:1111.3057}, 2011.

\bibitem{mevstrovic2013some}
R.~Me{\v{s}}trovic.
\newblock Some {W}olstenholme type congruences.
\newblock {\em Math. Appl}, 2:35--42, 2013.

\bibitem{slavutskii1999leudesdorf}
I.~S. Slavutskii.
\newblock Leudesdorf's theorem and {B}ernoulli numbers.
\newblock {\em Archivum Mathematicum}, 35(4):299--303, 1999.

\bibitem{sun2000congruences}
Z.-H. Sun.
\newblock Congruences concerning {B}ernoulli numbers and {B}ernoulli
  polynomials.
\newblock {\em Discrete Applied Mathematics}, 105(1-3):193--223, 2000.

\bibitem{sun2022series}
Z.-W. Sun.
\newblock Series with summands involving harmonic numbers.
\newblock In {\em Combinatorial and Additive Number Theory, New York Number
  Theory Seminar}, pages 413--460. Springer, 2022.

\bibitem{washington1998p}
L.~C. Washington.
\newblock $p$-adic {L}-functions and sums of powers.
\newblock {\em Journal of Number Theory}, 69(1):50--61, 1998.

\bibitem{wolstenholme1862certain}
J.~Wolstenholme.
\newblock On certain properties of prime numbers.
\newblock {\em QJ Math.}, 5:35--39, 1862.

\bibitem{zhao2007bernoulli}
J.~Zhao.
\newblock Bernoulli numbers, {W}olstenholme's theorem, and $p^5$ variations of
  {L}ucas' theorem.
\newblock {\em Journal of number theory}, 123(1):18--26, 2007.

\end{thebibliography}
\bibliographystyle{abbrv}

\end{document}